\documentclass[a4paper,reqno,11pt]{amsart}
\usepackage{amsmath,amssymb}
\usepackage{amsthm}
\usepackage{amscd}
\usepackage{amssymb}
\usepackage{amsfonts}
\usepackage{latexsym}
\usepackage[mathscr]{eucal}

\newtheorem{thm}{Theorem}[section]
\newtheorem{prop}[thm]{Proposition}
\newtheorem{lemma}[thm]{Lemma}

\newtheorem{defn}[thm]{Definition}

\begin{document}
\allowdisplaybreaks[3] 

\title{Perturbations of von Neumann subalgebras with finite index}

\author[S. Ino]{SHOJI INO}

\address{Department of Mathematical Sciences, Kyushu University, Motooka, Fukuoka, 819-0395, Japan.}

\email{s-ino@math.kyushu-u.ac.jp}


\begin{abstract}
In this paper, we study uniform perturbations of von Neumann subalgebras of a von Neumann algebra. 
Let $N$ and $M$ be von Neumann subalgebras of a von Neumann algebra with finite probabilistic index in the sense of Pimsner-Popa.
If $N$ and $M$ are sufficiently close, 
then $N$ and $M$ are unitarily equivalent.
The implementing unitary can be chosen as being close to the identity.
\end{abstract}

\maketitle

\section{introduction}

In 1972, the uniform perturbation theory of operator algebras was started by Kadison and Kastler \cite{KK}.
They defined a metric on the set of operator algebras on a fixed Hilbert space by the Hausdorff distance between their unit balls.
We get basic examples of close operator algebras by small unitary perturbations.
Namely, given an operator algebra $N \subset \mathbb{B}(H)$ and a unitary operator $u\in \mathbb{B}(H)$, 
if $u$ is close to the identity operator, then $u N u^*$ is close to $N$.
Conversely, Kadison and Kastler suggested that suitably close operator algebras must be unitarily equivalent.
This conjecture was solved positively for injective von Neumann algebras in \cite{Chris2,RT,Johnson1} with earlier special cases \cite{Chris1,P1}.
Recently, \cite{CCSSWW} gave classes of non-injective von Neumann algebras for which this conjecture valid.
In \cite{Chris3}, for von Neumann subalgebras in a {\it finite} von Neumann algebra, Kadison-Kastler conjecture was solved positively.
However, for general von Neumann algebras, this conjecture is still open.

Examples of non-separable C$^*$-algebras which are arbitrary close but non-isomorphic were found in \cite{CC}. 
However, for general separable C$^*$-algebras, Kadison-Kastler conjecture is still open.
In \cite{CSSWW}, the conjecture was solved positively for separable nuclear C$^*$-algebras.
Earlier special cases of \cite{CSSWW} was studied in \cite{Chris5,PR1,PR2,Khoshkam1}.
The author and Watatani showed that for an inclusion of simple C$^*$-algebras with finite index, sufficiently close intermediate C$^*$-subalgebras  are unitarily equivalent in \cite{IW}.
Although our constants depend on inclusions, 
Dickson got universal constants independent of inclusions in \cite{Dickson}.

In this paper, we study uniform perturbations of von Neumann subalgebras of a von Neumann algebra with finite index.
Let  $N$ and $M$ be von Neumann subalgebras of a von Neumann algebra $L$
 with conditional expectations $E_N:L\to N$ and $E_M:L\to M$ of finite probabilistic index in the sense of Pimsner-Popa \cite{PP}. 
If the distance between $N$ and $M$ is sufficiently small, then $N$ and $M$ are unitarily equivalent.
Moreover, the implementing unitary can be chosen as being close to the identity.
In general, there exist examples of arbitrarily close unitarily conjugate C$^*$-algebras where the implementing unitaries could not chosen to be close to the identity in \cite{Johnson2}.
Compared with the author and Watatani's C$^*$-algebraic case \cite{IW},
we do not assume that $N$ and $M$ have a common subalgebra with finite index.

\section{distance and the relative Dixmier property}

In this paper, all von Neumann algebras are countably decomposable, that is, they have faithful normal  states.

We recall the distance defined by Kadison and Kastler in \cite{KK} and near inclusions defined by Christensen in \cite{Chris5}.
For a von Neumann algebra $N$,
we denote by $N_1$ the unit ball of $N$.

\begin{defn}\label{metric}\upshape
Let $N$ and $M$ be von Neumann algebras in $\mathbb{B}(H)$.
Then, the distance between $N$ and $M$ is defined by
\[
d(N,M):= \max \left\{ \sup_{n\in N_1} \inf_{m\in M_1} \| n-m \| \ , \ \sup_{m\in M_1} \inf_{n\in N_1} \| m-n\| \right\}.
\]
Let $\gamma>0$.
We say that $N$ is $\gamma$ contained in $M$ and write $N \subseteq_{\gamma} M$ if
for any $n\in N_1$, there exists $m\in M$ such that $\| n - m\| \le \gamma$.
\end{defn}

If $d(N,M)<\gamma$, then for any $x$ in either $N_1$ or $M_1$, there exists $y$ in the other unit ball such that $\| x-y\| \le \gamma$.

The following well-known fact is needed to show that maps are onto in Proposition \ref{isomorphism}.

\begin{lemma}\label{lem:<1}
Let $N$ and $M$ be von Neumann algebras in $\mathbb{B}(H)$. 
If $N\subset M$ and $d(N ,M )<1$, then $N=M$. 
\end{lemma}

The next lemma records some standard estimates.

\begin{lemma}\label{lem:estimates}
Let $A$ be a unital $\mathrm{C}^*$-algebra.
\begin{enumerate}
\item 
Let $x \in A$ satisfy that $\|x-I\|<1$ and 
$u \in A$ be the unitary in the polar decomposition $x = u|x|$. 
Then,
\[ 
\|u-I\| \le \sqrt{2}\|x-I\|. 
\]
\item
Let $p$ and $q$ be projections in $A$ with $\|p-q\|<1$. 
Then, there exists a unitary $w\in A$ such that 
\[
w p w^*=q  \ \ \text{ and } \ \  
\|w-I\|\le \sqrt{2}\|p-q\|.
\]
\end{enumerate}
\end{lemma}

Jones introduced index for inclusions of type II$_1$ factors in \cite{Jones}.
For arbitrary factors, Kosaki extended Jones' notion of index in \cite{Kosaki}.
The following definition was introduced by Pimsner and Popa in \cite{PP}.

\begin{defn}\upshape
Let $N\subset M$ be an inclusion of von Neumann algebras and $E:M\to N$ be a conditional expectation.
Then, we call $E$ is {\it of finite probabilistic index} if there exists $c\ge 0$ such that
\[
 E(x^*x) \ge c x^*x \ \ \ \text{for} \ x\in M.
\]
When $E$ is of finite probabilistic index, we define the probabilistic index of $E$ by $( \sup \{ c\ge 0 \, | \,  E(x^*x) \ge  c x^* x  \ \, \text{for} \ x \in M \} )^{-1}$.
\end{defn}

We recall the basic construction (see \cite{Popa1}).
Let $N\subset M$ be an inclusion of von Neumann algebras with  a faithful normal conditional expectation $E_N:M\to N$
and $\psi$ be a faithful normal sate on $N$. 
Put $\phi:=\psi\circ E_N$.
Then, $\phi$ is a faithful normal state on $M$.
Let $(H, \pi, \xi)$ be the GNS triplet associated with $\phi$.
Then, we get the Jones projection $e_N \in \mathbb{B}(H)$ satisfying
\[
{\rm Im}(e_N)=[N\xi]   \ \ \ \text{and} \ \ \ e_N (x \xi)= E_N(x) \xi \ \ \text{for} \ x\in M. 
\]
The basic construction $\langle M, e_N \rangle$ is the von Neumann algebra in $\mathbb{B}(H)$ generated by $M$ and $e_N$.
If $E_N$ is of finite probabilistic index, then there exists a conditional expectation $E_M: \langle M,e_N \rangle \to M$ of finite probabilistic index by \cite{Popa1}.

\

Let $N\subset M$ be an inclusion of von Neumann algebras.
For any $x\in M$, we will denote by $C_N(x)$ the norm closure of the convex hull of $\{ u x u^* \, | \, u \ \text{is unitary element in} \ N \}$.
We recall the relative Dixmier property for inclusions of von Neumann algebras after Popa \cite{Popa2}.

\begin{defn}\upshape
Let $N\subset M$ be an inclusion of von Neumann algebras.
Then, we say that $N\subset M$ has {\it the relative Dixmier property} if for any $x\in M$,
$C_N(x) \cap N'\cap M \neq \emptyset$.
\end{defn}

In \cite{Popa2}, Popa showed the following theorem.

\begin{thm}[Popa \cite{Popa2}]\label{Popa}
Let $N\subset M$ be an inclusion of von Neumann algebras with a conditional expectation $E:M\to N$ of finite probabilistic index.
Then, $N\subset M$ has the relative Dixmier property.
\end{thm}

We shall show relations between the relative Dixmier property and the distance.

Let $N\subset M$ be an inclusion of von Neumann algebras.
For any $x\in M$, the map $ad(x) : N\to M$ is defined by $(ad(x))(y)=y x- x y$.

The proof of the next proposition follows from  \cite[Proposition 2.5]{CSSW}.

\begin{prop}\label{distance of commutant}
Let $N$ and $M$ be von Neumann subalgebras of a von Neumann algebra $L$  with $N\subseteq_{\gamma}M$.
If $N\subset L$ has the relative Dixmier property, then
\[
M'\cap L\subseteq_{2 \gamma} N'\cap L.
\]
\end{prop}

\begin{proof}
For any $x\in M'\cap L_1$, there exists $y\in C_N(x)\cap N'\cap M$.
Since for any unitary $u\in N$,
\[
\|u x u^*-x\| =\|u x-x u\| =\|(ad(x))(u)\| \le \| ad(x)\|,
\]
we have $\|y-x\|\le \|ad(x)\|$.
On the other hand, for any $n\in N_1$, there exists $m\in M$ such that $\| n- m\|\le \gamma$.
Thus,
\begin{align*}
\| (ad(x))(n)\|
&= \|n x - x n\| = \| n x - m x + x m - x n \| \\
&\le \| n - m \| + \| m - n \| \le 2\gamma.
\end{align*}
Namely, $\| x -y \| \le \|ad(x) \| \le 2\gamma$.
\end{proof}

\section{perturbations}

In the following proposition, we construct a surjective $*$-isomorphism between von Neumann subalgebras of a von Neumann algebra with finite probabilistic index.
The argument is originated in early work of Christensen \cite{Chris2,Chris3}.

\begin{prop}\label{isomorphism}
Let  $N$ and $M$ be von Neumann subalgebras of a von Neumann algebra $L$ 
with conditional expectations $E_N:L\to N$, $E_M:L\to M$ of finite probabilistic index.
If $d(N,M)<1/15$, then there exists a normal surjective $*$-isomorphism $\Phi:N\to M$ such that $\| \Phi-id_N\| < 14 d(N,M)$.
\end{prop}

\begin{proof}
Put $\gamma:= (1.01) d(N,M)$.
Let $\langle L, e_M\rangle$ be the basic construction by using $E_M:L\to M$.
Then, there exists a conditional expectation $E_L: \langle L, e_M \rangle \to L$ of finite probabilistic index.
Since $E_N\circ E_L:\langle L, e_M\rangle \to N$ is of finite probabilistic index, 
$N\subset \langle L, e_M\rangle$ has the relative Dixmier property by Theorem \ref{Popa}.
Therefore, 
\[
M'\cap \langle L, e_M\rangle \subseteq_{2\gamma} N'\cap \langle L, e_M\rangle
\]
by Proposition \ref{distance of commutant}.
Thus, there exists $t\in N'\cap \langle L, e_M\rangle$ such that $\|t- e_M\|\le 2 \gamma <1/2$.
Put $p:= \chi_{[1-2\gamma,\, 1+2\gamma]}((t+t^*)/2)$.
Since we have $\|p-e_M\| \le \| p- t\| +\| t -e_M \|\le 4\gamma <1$, there exists a unitary $w\in \langle L, e_M\rangle$ such that 
\[
we_M w^*=p \ \ \ \text{and} \ \ \   \|w-I\|\le 4\sqrt{2} \gamma
\]
by Lemma \ref{lem:estimates}.
For any $x\in N$, we define $\tilde\Phi(x):=e_M w^* x w e_M= w^* p x p w$.
Then, $\tilde\Phi:N \to e_M \langle L, e_M \rangle e_M$ is a normal $*$-homomorphism, because $p\in N'$.
Now, there exists a surjective $*$-isomorphism $\iota: e_M \langle L, e_M \rangle e_M \to M$.
Hence, we can define a normal $*$-homomorphism $\Phi:=\iota \circ \tilde\Phi:N\to M$.
For any $x\in N_1$, 
\begin{align*}
\| \Phi(x)-E_M(x)\|
&= \|e_M ( \Phi(x) - E_M(x) )e_M\| \\
&= \| e_M w^* x w e_M - e_M x e_M\| \\
&\le 2\| w-I\| \le 8\sqrt{2} \gamma.
\end{align*}
Therefore, by \cite[Lemma 3.2]{IW},
\[
\| \Phi-id_N\| \le \|\Phi- E_M|_N\| +\| E_M|_N -id_N\| \le (8\sqrt{2}+2)\gamma < 14 d(N,M) <1.
\]
This gives that $\Phi$ is a $*$-isomorphism.

Moreover, for any $x\in M_1$,  there exists $y\in N_1$ such that $\|x-y\| \le \gamma$.
Then,
\begin{align*}
\| x- \Phi(y)\|
& \le \|x-y\|+\| y- \Phi(y)\| \\
&\le \gamma +(8\sqrt{2}+2) \gamma < 15 d(N,M) <1.
\end{align*}
Since this gives that $d(M, \Phi(N))<1$, $\Phi(N)=M$ by Lemma \ref{lem:<1}. 
\end{proof}

The following is our main theorem in this paper.
We based on Christensen's work \cite[Proposition 4.2]{Chris2}
and \cite[Proposition 3.2]{Chris3}.

\begin{thm}
Let  $N$ and $M$ be von Neumann subalgebras of a von Neumann algebra $L$ 
with conditional expectations $E_N:L\to N$, $E_M:L\to M$ of finite probabilistic index.
If $d(N,M)<1/15$, then there exists a unitary $u\in L$ such that $N =u Mu^*$ and
$\|u-I\| < 20 d(N,M)$.
\end{thm}

\begin{proof}
By Proposition \ref{isomorphism}, there exists a normal surjective $*$-isomorphism $\Phi:N\to M$ such that 
$\|\Phi-id_N\| < 14 d(N,M)$.
Put
\[
K:=\left\{ 
\begin{pmatrix}
x &0 \\
0 &\Phi(x)
\end{pmatrix}
\, \Big{|} \ x\in N
\right\}.
\]
Then, we can define a conditional expectation $E_K: \mathbb{M}_2(L)\to K$ of finite probabilistic index
by
\[
E_K \left( 
\begin{pmatrix}
a & b \\
c & d
\end{pmatrix}
\right)
=
\begin{pmatrix}
\frac{E_N(a)+\Phi^{-1}(E_M(d))}{2} & 0 \\
0 & \frac{\Phi(E_N(a))+E_M(d)}{2}
\end{pmatrix}
 \  \  \ \text{for} \ 
\begin{pmatrix}
a & b \\
c & d
\end{pmatrix}
\in \mathbb{M}_2(L).
\]
Therefore, $K\subset \mathbb{M}_2(L)$ has the relative Dixmier property by Theorem \ref{Popa}.
Applying the relative Dixmier property for $ \bigl(\begin{smallmatrix} 0 & I \\ 0 & 0 \end{smallmatrix} \bigr) \in \mathbb{M}_2(L) $,
we obtain $x$ in $C_K
\left( \bigl(\begin{smallmatrix} 0 & I \\ 0 & 0 \end{smallmatrix} \bigr) \right)
\cap K'\cap \mathbb{M}_2(L)$.
Then, there exists $y\in L$ such that $x=
\bigl(\begin{smallmatrix} 0 & y \\ 0 & 0 \end{smallmatrix}\bigr)
$,
because for any unitary $u\in N$,
\[
\begin{pmatrix}
u & 0 \\
0 & \Phi(u)
\end{pmatrix}
\begin{pmatrix}
0 & I \\
0 & 0
\end{pmatrix}
\begin{pmatrix}
u^* & 0 \\
0 & \Phi(u^*)
\end{pmatrix}
=
\begin{pmatrix}
0 & u\Phi(u^*) \\
0 & 0
\end{pmatrix}.
\]
Furthermore,
\[
\|y-I\|\le \sup_{u\in N^u} \| u\Phi(u^*)-I\| =\sup_{u\in N^u}\| \Phi(u^*)-u^* \| \le \|\Phi-id_N\| <1 .
\]
By Lemma \ref{lem:estimates}, the unitary $u\in L$ in the polar decomposition $y=u|y|$ 
satisfies 
\[
\| u-I\|\le \sqrt{2} \| \Phi-id_N\| < 20 d(N,M).
\]
Since $x= 
\bigl(\begin{smallmatrix} 0 & y \\ 0 & 0 \end{smallmatrix}\bigr)
\in K'$, for any $n\in N$,
\[
\begin{pmatrix}
0 & y\Phi(n) \\
0 & 0
\end{pmatrix}
=\begin{pmatrix}
0 & y \\
0 & 0
\end{pmatrix}
\begin{pmatrix}
n & 0 \\
0 & \Phi(n)
\end{pmatrix} =\begin{pmatrix}
n & 0 \\
0 & \Phi(n)
\end{pmatrix}
\begin{pmatrix}
0 & y \\
0 & 0
\end{pmatrix} 
=\begin{pmatrix}
0 & n y \\
0 & 0
\end{pmatrix}.
\]
By taking adjoints, we have
\[
\Phi(n) y^*= y^* n  \ \ \ \text{for} \ n\in N.
\]
Therefore,
\[
y^*y \Phi(n)= y^* n y= \Phi(n) y^* y \ \ \ \text{for} \ n\in N.
\]
For any $n\in N$, since this gives $|y| \Phi(n)=\Phi(n) |y|$, $u \Phi(n)=n u$.
Hence, $u M u^*=u \Phi(N) u^*=N $.
\end{proof}

\end{document}